\documentclass[12pt,a4paper]{article}

\usepackage{graphicx}
\usepackage{amsmath}
\usepackage{amsthm}
\usepackage{amssymb}
\usepackage{amsfonts}
\usepackage{latexsym}
\usepackage{epsfig}
\usepackage{color}

\newcommand{\beql}[1]{\begin{equation}\label{#1}}
\newcommand{\eeq}{\end{equation}}
\newcommand{\comment}[1]{}

\newcommand{\eref}[1]{{\rm (\ref{#1})}}

\newcommand{\Abs}[1]{{\left|{#1}\right|}}

\newcommand{\Norm}[1]{{\left\|{#1}\right\|}}

\newcommand{\Set}[1]{{\left\{{#1}\right\}}}

\newcommand{\RR}{{\mathbb R}}
\newcommand{\CC}{{\mathbb C}}
\newcommand{\ZZ}{{\mathbb Z}}
\newcommand{\NN}{{\mathbb N}}
\newcommand{\TT}{{\mathbb T}}

\newcommand{\BB}{{\mathcal B}}

\newcommand{\ve}{{\varepsilon}}

\newcommand{\inner}[2]{{\left\langle #1, #2 \right\rangle}}

\newcommand{\dens}{{\rm dens\,}}

\newcommand{\supp}{{\rm supp\,}}
\newcommand{\dist}{{\rm dist\,}}

\newcommand{\ft}[1]{\widehat{#1}}

\newcounter{open}
\newcounter{dfn}
\def\thedfn{\arabic{dfn}}

\newcounter{obs}
\def\theobs{\arabic{obs}}

\newcounter{thm}
\newcounter{othm}
\def\theothm{\Alph{othm}}

\newcounter{mysec}
\newcounter{mysubsec}[mysec]
\newtheorem{theorem}{Theorem}
\newtheorem{corollary}{Corollary}
\newtheorem{lemma}{Lemma}
\newtheorem{proposition}{Proposition}
\newtheorem{definition}{Definition}
\newtheorem{conjecture}{Conjecture}

\newtheorem{problem}{Problem}
\theoremstyle{definition}
\newtheorem{example}{Example}
\newtheorem{remark}{Remark}

\newcommand{\Tu}{{\mathcal T}}
\newcommand{\Om}{{\Omega}}
\newcommand{\FF}{{\mathcal F}}

\newcommand{\e}{\varepsilon}

\newcounter{rev}
\newcounter{rem}
\begin{document}

\title{Tur\'an's extremal problem on locally compact abelian groups}

\author{Szil\' ard Gy. R\' ev\' esz\thanks{Supported in part by the Hungarian National
Foundation for Scientific Research, Project \#s T-49301, K-61908
and K-72731.}}

\let\oldfootnote\thefootnote
\def\thefootnote{}

\let\thefootnote\oldfootnote

\maketitle

{\small \tableofcontents}

\newpage

\begin{abstract}
Let $G$ be a locally compact abelian group (LCA group) and
$\Omega$ be an open, 0-symmetric set. Let $\FF:=\FF(\Omega)$ be
the set of all continuous functions $f:G\to\RR$ which are
supported in $\Omega$ and are positive definite. The Tur\'an
constant of $\Omega$ is then defined as $\Tu(\Om):=\sup\{
\int_{\Om}f~:~f\in\FF(\Om), f(0)=1\}$.

Mihalis Kolountzakis and the author has shown that structural
properties -- like spectrality, tiling or packing with a certain
set $\Lambda$ -- of subsets $\Om$ in finite, compact or Euclidean
(i.e., $\RR^d$) groups and in $\ZZ^d$ yield estimates of
$\Tu(\Om)$. However, in these estimates some notion of the size,
i.e. density of $\Lambda$ played a natural role, and thus in
groups where we had no grasp of the notion, we could not
accomplish such estimates.

In the present work a recent generalized notion of uniform
asymptotic upper density is invoked, allowing a more general
investigation of the Tur\'an constant in relation to the above
structural properties. Our main result extends a result of Arestov
and Berdysheva, (also obtained independently and along different
lines by Kolountzakis and the author), stating that convex tiles
of a Euclidean space necessarily have $\Tu_{\RR^d}(\Om)
=|\Om|/2^{d} $. In our extension $\RR^d$ could be replaced by any
LCA group, convexity is dropped, and the condition of tiling is
also relaxed to a certain packing type condition and positive
uniform asymptotic upper density of the set $\Lambda$.

Also our goal is to give a more complete account of all the
related developments and history, because until now an exhaustive
overview of the full background of the so-called Tur\'an problem
was not delivered.
\end{abstract}



\newpage

\section{Introduction} \label{sec:introduction}

\subsection{The Tur\'an problem}

We study the following problem, generally investigated under the
name of \emph{''Tur\' an's Problem"}, following Stechkin
\cite{stechkin:periodic}, who recalls a question posed to him in
personal discussion.

\begin{problem}\label{pr:Turan}
Given an open set $\Omega$, symmetric about $0$, and a continuous,
positive definite, integrable function $f$, with $\supp f
\subseteq \Omega$ and with $f(0)=1$, how large can $\int f$ be?
\end{problem}

Although this name for the problem is quite widespread, one has to
note that all the important versions of the problem were
investigated well before the beginning of the seventies, when the
discussion of Tur\'an and Stechkin took place.

About the same time when Tur\'an discussed the question with
Stechkin, American researchers already investigated in detail the
\emph{square integral version of the problem}, see \cite{GRR,
Page, Domar}. Their reason for searching the extremal function and
value came from radar engineering problems at the Jet Propulsion
Laboratory.

More importantly, Problem \ref{pr:Turan} appears as early as in
the thirties \cite{Siegel}, when Siegel considered the question
for $\Omega$ being a ball, or even an ellipsoid in Euclidean space
$\RR^d$, and established the right extremal value $|\Omega|/2^d$.
The question occurred to Siegel as a theoretical possibility to
sharpen the Minkowski Latice Point Theorem. Although Siegel
concluded that, due to the extremal value being just as large as
the Minkowski Lattice Point Theorem would require, this geometric
statement can not be further sharpened through improvement on the
extremal problem, nevertheless he works out the extremal problem
fully and exhibits some nice applications in the theory of entire
functions.

Furthermore, the same Problem \ref{pr:Turan} appeared in a paper
of Boas and Kac \cite{BK} in the forties, even if the main
direction of the study there was a different version, what is
nowadays generally called the \emph{pointwise Tur\'an problem}.
However, as is realized partially in \cite{BK} and fully only
later in \cite{kolountzakis:pointwise}, the pointwise Tur\'an
problem -- formulated in the classical setting of Fourier series,
but nevertheless equivalent to the Euclidean space settings of
\cite{BK} -- goes back already to Caratheodory \cite{Cara} and
Fej\'er \cite{Fej}.

The Tur\'an problem was considered by Stechkin on an interval in
the torus $\TT=\RR/\ZZ$ \cite{stechkin:periodic} and in $\RR$ by
Boas and Kac \cite{BK}, but extensions were to follow in several
directions.

Such a question is interesting in the study of sphere packings
\cite{gorbachev:sphere,cohn:packings, ConSlo}, in additive number
theory \cite{ruzsa:uniform,KMF,Mont,GorbMan03} and in the theory
of Dirichlet characters and exponential sums \cite{KS}, among
other things.

\subsection{One dimensional case of the Tur\'an problem}
\label{sec:onedim}

Already the symmetric interval case in one dimension presents
nontrivial complications, which were resolved satisfactorily only
recently. We discuss the development of the problem from the
outset to date.

Actually, Tur\'an's interest might have come from another area in
number theory, namely Diophantine approximation. (Let us point out
that \cite{AKP} starts with the sentence: "With regard to
applications in number theory, P. Tur\'an stated the following
problem:", while at the end of the paper there is special
expression of gratitude to Professor Stechkin for his interest in
this work. Also, Gorbachev writes in \cite[p.
314]{gorbachev:ball}): "Studying applications in number theory, P.
Tur\'an posed the problem ...")

One can hypothesise that Tur\'an thought of the elegant proof of
the well-known Dirichlet approximation theorem, stating that for
any given $\alpha\in\RR$ at least one multiple $n\alpha$ in the
range $n=1,\dots,N$ have to approach some integer as close as
$1/(N+1)$. The proof, which uses Fourier analysis and Fej\'er
kernels in particular, is presented in \cite[p. 99]{Mont}, and in
a generalized framework it is explained in \cite{BlMo}, but it is
remarked in \cite[p. 105]{Mont} that the idea comes from Siegel
\cite{Siegel}, so Tur\'an could have been well aware of it. Let us
briefly present the argument right here.

If we wish to detect multiples $n\alpha$ of $\alpha\in \RR$ which
fall in the $\delta$-neighborhood of an integer, that is which
have $\|n\alpha\| < \delta$ (where, as usual in this field,
$\|x\|:=\dist (x,\ZZ)$), then we can use that for the triangle
function $F(x):=F_{\delta}(x):=\max (1-\|x\|/\delta)_{+}$, we have
$F(n\alpha)>0$ iff $\|n\alpha\|<\delta$. So if with an arbitrary
$\delta>1/(N+1)$ we can work through a proof of $F(n\alpha)>0$ for
some $n\in [1,N]$, then the proof yields the sharp form of the
Dirichlet approximation theorem. (It is indeed sharp, because for
no $N\in\NN$ can any better statement hold true, as the easy
example of $\alpha:=1/(N+1)$ shows.)

So we take now $S:=\sum_{n=1}^{N}(1-\frac{|n|}{N+1}) F(n\alpha)$,
or, since $F$ is even and $F(0)=1$, consider the more symmetric
sum $2S+1=\sum_{n=-N}^{N}(1-\frac{|n|}{N+1}) F(n\alpha)$. Note
that $\widehat{F_ \delta}(t)=\delta \cdot\left(\frac{\sin(\pi
\delta t)}{\pi \delta t}\right)^2$, so in particular with the
nonnegative coefficients $\widehat{F}(k)=c_k$ we can write (with
$e(t):=e^{2\pi i t}$)
\begin{equation}\label{eq:triangleFs}
F_{\delta}(x)=\sum_{k=-\infty}^{\infty} c_k e(kx)\qquad
c_0=\delta,~c_k=\delta\cdot\left(\frac{\sin(\pi k\delta)}{\pi
k\delta}\right)^2 ~(k=\pm 1,\pm 2,\dots).
\end{equation}
It suffices to show $S>0$. With the Fej\'er kernels
$\sigma_N(x):=\sum_{n=-N}^N \left(1-\frac{|n|}{N+1}\right)
e(nx)=\frac{1}{N+1} \cdot\left(\frac{\sin(\pi (N+1)x)}{\pi
x}\right)^2\geq 0$, after a change of the order of summation we
are led to
\begin{align*}
2S+1 & = \sum_{k=-\infty}^{\infty} c_k
\sum_{n=-N}^{N}\left(1-\frac{|n|}{N+1}\right) e(nk\alpha) \\ &
=c_0 \sigma_N(0) + 2 \sum_{k=1}^{\infty} c_k \sigma_N(k\alpha)\geq
c_0\sigma_N(0)= \delta (N+1) >1,
\end{align*}
which concludes the argument.

Now if in place of the triangle function with $\delta=1/(N+1)$
another positive definite (i.e. $\widehat{f}\geq 0$) function $f$
could be put with $\supp f \subset [-\delta,\delta]$ and $f(0)=1$
but with $\widehat{f}(0)>\delta$ then the above argument with $f$
in place of $F$ would give $S>0$ even for $\delta=1/(N+1)$,
clearly a contradiction since the Dirichlet approximation theorem
cannot be further sharpened. That round-about argument already
gives that for $h$ a reciprocal of an integer, the triangle
function $F_h$ is extremal in the Tur\'an problem for $[-h,h]$. In
other words, we obtain Stechkin's result \cite{stechkin:periodic},
(see also below) already from considerations of Diophantine
approximation.

So Tur\'an asked Stechkin if for any $h>0$ the triangle function
provides the largest possible integral among all positive definite
functions vanishing outside $[-h,h]$ and normalized by attaining
the value $1$ at $0$. Stechkin derived that this is the case for
$h$ being the reciprocal of a natural number: by monotonicity in
$h$ for other values he could conclude an estimate. Anticipating
and slightly abusing the general notations below, denote the
extremal value by $T(h)$: then Stechkin obtained $T(h)=h+O(h^2)$.
This was sharpened later by Gorbachev \cite{gorbachev:ball} and
Popov \cite{Popov} (cited in \cite[p. 77]{IGR}) to $h+O(h^3)$.

The corresponding Tur\'an extremal value $T_{\RR}(h)$ on the real
line is, by simple dilation, depends linearly on the interval
length and is just $h T_{\RR}(1)$ for any interval $I=[-h,h]$. On
the other hand it follows already from $\lim_{h\to 0+} T(h)/h=1$
that e.g. for the unit interval $[-1,1]$ the extremal function
must be the triangle function and $T_{\RR}(1)=1$, hence
$T_{\RR}(h)=h$. In fact, this case was already settled earlier by
Boas and Katz in \cite{BK} as a byproduct of their investigation
of the pointwise question.

But there is another observation, seemingly well-known although no
written source can be found. Namely, it is also known for some
time that for $h$ \emph{not being a reciprocal of an integer
number}, the triangle function \emph{can indeed be improved upon}
a little. Indeed, the triangle function $F_h$ has Fourier
transform which vanishes precisely at integer multiples of $1/h$,
and in case $1/h \notin \NN$, some multiples fall outside $\ZZ$.
And then the otherwise double zeroes of $\widehat{F_h}$ can even
be substituted by a product of two close-by zero factors, allowing
a small interval in between, where the Fourier transform can be
negative. This negativity spoils positive definiteness regarding
the function on $\RR$: but on $\TT$ it does not, for only the
values at integer increments must be nonnegative in order that a
function be positive definite on $\TT$. With a detailed calculus
(using also the symmetric pair of zeroes) such an improvement upon
the triangle function is indeed possible. (Note that here
$\widehat{F}$, so also $\int \widehat{F}=F(0)$ is perturbed while
$\widehat{F}(0)=\int F$ is unchanged.) I have heard this
construction explained in lectures during my university studies
\cite{Hal}; in Russia, a similar observation was communicated by
A. Yu Popov \cite{Popov} and later recorded in writing in
\cite{GorbMan01, GorbMan04, IGR}.

As said above, the computation of exact values of $T(h)$ started
with Stechkin for $h=1/q$, $q\in \NN$: these are the only cases
when $T(h)=h$. Further values, already deviating from this simple
formula, were computed for {\em some rational} $h$ in
\cite{Manoshina, GorbMan01,GorbMan04} and finally {\em for all
rational} $h$ in \cite{IGR,IvanovRudomazina}. Knowing the value
for rational $h$ led Ivanov to further investigations which
established continuity of the extremal value in function of $h$,
and thus gave the complete solution of Tur\'an's problem on the
torus \cite{Ivanovfull}. In fact, the above works also established
that for $[-h,h]\subset \TT$ the Tur\'an extremal problem and the
Delsarte extremal problem (described below in \S
\ref{sec:variants}) has the same extremal value (and extremal
functions). Note that this coincidence does not hold true in
general.

However, it seems that almost nothing is known about Tur\'an
extremal values of other, one would say "dispersed" sets not being
intervals. A natural conjecture is that e.g. on $\RR$ (or perhaps
even on $\TT$ ?) a set $\Omega\subset\RR$ of fixed measure
$|\Omega|=m$ can have maximal Tur\'an constant value if only it is
a zero-symmetric interval $[-m/2,m/2]$. What we know from
\cite[Theorem 6]{kolountzakis:groups} is that we certainly have
$T(\Omega) \leq m/2$, that is, in $\RR$ no ''better sets", than
zero-symmetric intervals, can exist. However, uniqueness is not
known, not even for $\RR$. In \cite{kolountzakis:groups} there is
a more general estimate in function of the prescribed measure $m$,
but for higher dimensions it is far less precise. Also, regarding
the discrete group $\ZZ$ one must observe that zero-symmetric
intervals $[-N,N]\subset \ZZ$ have the same Tur\'an extremal
values as their homothetic copies $k[-N,N]$ ($k\in\NN$) which
already destroys the hope for ''uniqueness only for intervals". In
higher dimensions not even the right class of the corresponding
''condensed sets", like intervals in dimension one, has been
identified.

\subsection{Tur\'an's problem in the multivariate setting}
\label{sec:multivariate}

Already as early as in the 1930's, Siegel \cite{Siegel} proved
that for an ellipsoid in $\RR^d$ the extremal value in Problem
\ref{pr:Turan} is $|\Om|/2^d$.

In the 1940's, Boas and Katz \cite{BK} mentioned that Poisson
summation may be used to treat similar questions in higher
dimensions. Besides mentioning the group settings, Garcia \& al.
\cite{GRR} and Domar \cite{Domar} also touches upon the question
without going into further details. The packing problem by balls
in Euclidean space has already been treated by many authors via
multivariate extremal problems of the type, but there the optimal
approach is to pose a closely related, still different variant,
named Delsarte- (and also as Logan- and Levenshtein-) problem. See
e.g. \cite{gorbachev:sphere, cohn:packings} and the references
therein.

As a direct generalization of Stechkin's work, Andreev
\cite{Andreev} calculated the Tur\'an constants of cubes $Q_h^d$
in $\TT^d$ obtaining $h^d+O(h^{d+1})$. Moreover, he estimated the
Tur\'an constant of the cross-politope ($\ell_1$-ball) $O_h^d$ in
$\TT^d$: his estimates are asymptotically sharp when $d=2$.
Gorbachev \cite{gorbachev:ball} simultaneously sharpened and
extended these results proving that for any centrally symmetric
body $D\subset [-1,1]^d$ and for all $0<h<1/2$ we always have
$\Tu_{\TT^d}(hD)=\Tu_{\RR^d}(D)\cdot h^d+O(h^{d+2})$.

Arestov and Berdysheva \cite{hexagon} offers a systematic
investigation of the multivariate Tur\'an problem collecting
several natural properties. They also prove that the hexagon has
Tur\'an constant exactly one fourth of the area of itself.
Gorbachov \cite{gorbachev:ball} proved that the unit ball
$B_d\subset \RR^d$ has Tur\'an constant $2^{-d}|B_d|$, where
$|B_d|$ is the volume ($d$-dimensional Lebesgue measure) of the
ball. Another proof of this fact can be found in
\cite{kolountzakis:turan}, but we have already noted that the
result goes back to Siegel \cite{Siegel}.

There is a special interest in the case which concerns $\Omega$
being a (centrally symmetric) convex subset of $\RR^d$
\cite{hexagon,arestov:tiles,gorbachev:ball,kolountzakis:turan},
since in this case the natural analog of the triangle function,
the self-convolution (convolution square) of the characteristic
function $\chi_{\frac12 \Omega}$ of the half-body $\frac12 \Omega$
is available showing that $\Tu_{\RR^d}(\Omega) \geq |\Omega|/2^d$.
The natural conjecture is that for a symmetric convex body this
convolution square is extremal, and $\Tu_{\RR^d}(\Omega) =
|\Omega|/2^d$. (Note that this fails in $\TT^d$, already for
$d=1$, for some sets $\Omega$.) Convex bodies with this property
may be called Tur\'an type, or Stechkin-regular, or, perhaps,
\emph{Stechkin-Tur\'an domains}, while symmetric convex bodies in
$\RR^d$ with $\Tu_{\RR^d}(\Omega) > |\Omega|/2^d$ as anti-Tur\'an
or \emph{non-Stechkin-Tur\'an} domains. Thus the above mentioned
result about the ball can be reworded saying that the ball is of
Stechkin-Tur\'an type.

To date, no non-Stechkin-Tur\'an domains are known, although the
family of known Stechkin-Tur\'an domains is also quite meager
(apart from $d=1$ when everything is clear for the intervals).

In \cite{hexagon,arestov:tiles} Arestov and Berdysheva prove that
if $\Omega\subseteq\RR^d$ is a convex polytope which can tile
space when translated by the lattice $\Lambda\subseteq\RR^d$ (this
means that the copies $\Omega+\lambda$, $\lambda\in\Lambda$, are
non-overlapping and almost every point in space is covered) then
$\Tu_{\RR^d}(\Omega) = \Abs{\Omega}/2^{d}$. Whence the class of
Stechkin-Tur\'an domains includes, by the result of Arestov and
Berdysheva, convex lattice tiles.

Kolountzakis and R\'ev\'esz \cite{kolountzakis:turan} showed the
same formula for all convex domains in $\RR^d$ which are
\emph{spectral}. For the definition and some context see  \S
\ref{sec:spectral}, where it will be explained that all convex
tiles are spectral, and so the result of Arestov and Berdysheva is
also a consequence of the result in \cite{kolountzakis:turan}.

For not necessarily convex sets, further results are contained in
\cite{kolountzakis:groups} for $\RR^d$, $\TT^d$ and $\ZZ^d$.

\subsection{Variants and relatives of the Tur\'an problem}
\label{sec:variants}

In the same class of functions $\FF$ various similar quantities
may be maximized. The two most natural versions concern the
\emph{square-integral} of $f\in\FF$, henceforth called the
\emph{square-integral Tur\'an problem}, and the \emph{function
value} at some arbitrarily prescribed point $z\in\Om$, called the
\emph{pointwise Tur\'an problem}.

The square-integral Tur\'an problem occurred for applied
scientists in connection with radar design (radar ambiguity and
overall signal strength maximizing), see \cite{Page, GRR}. Further
interesting results were obtained in \cite{Domar}. Nevertheless,
already on the torus $\TT$ the exact answer is not known, even if
Page \cite{Page} provides convincing computational evidence for
certain conjectures in case $h=\pi/n$, and the existence of
\emph{some} extremal function is known.

The natural pointwise analogue of Problem \ref{pr:Turan} is the
maximization of the function value $f(z)$, for given, fixed $z\in
\Om$, in place of the integral, over functions from the same class
than in Problem \ref{pr:Turan}. (Actually, the question can as
well be posed in any LCA group.) For intervals in $\TT$ or $\RR$
this was studied in \cite{ABB} under the name of "the pointwise
Tur\'an problem", although the same problem was already settled in
the relatively easy case of an interval $(-h,h)\subset\RR$ by Boas
and Kac in \cite{BK}. For general domains in arbitrary dimension
the problem was further studied in \cite{kolountzakis:pointwise}.

Further ramifications are obtained with considering different
variations of the above definitions. E.g. Belov and Konyagin
\cite{Belov-Konyagin:summary, Belov-Konyagin:detailed} considers
functions with integer coefficients, and periodic even functions
$f\sim \sum_k a_k \cos(kx)$ with $\sum_k |a_k| =1$ but with not
necessarily $a_k\geq 0$, i.e. not necessarily positive definite.

Berdysheva and Berens considers the multivariate question
restricted to the class of $\ell_1$-radial functions \cite{BerBer}.

A very natural version of the same problem is the Delsarte problem
\cite{Del} (also known under the names of Logan and Levenshtein):
here the only change in the conditioning of the extremal problem
is that we assume, instead of vanishing of $f$ outside a given set
$\Omega$, only the less restrictive condition that $f$ be
nonnegative outside the given set. Both extremal problems are
suitable in deriving estimates of packing densities through
Poisson summation: this is exploited in particular for balls in
Euclidean space, see e.g. \cite{Del, KabLev, Lev, AreBab97,
ConSlo, AreBab00, gorbachev:sphere, cohn:packings}.

There are several other rather similar, yet different extremal
problems around. E.g. one related intriguing question \cite{SS},
dealt with by several authors, is the maximization of $\int f$ for
real functions $f$ supported in $[-1,1]$, admitting
$\|f\|_{\infty}=1$, but instead of being positive definite, (which
in $\RR$ is equivalent to being represented as $g*\widetilde{g}$),
having only a representation $f=g*g$ with some $g\geq 0$ supported
in the half-interval $[-1/2,1/2]$.

Here we do not consider these relatives of the Tur\'an problem.

\subsection{Extension of the problem to LCA groups}
\label{sec:Turanongroups}

Some authors have already extended the investigations, although
not that systematically as in case of the multivariate setting, to
locally compact abelian groups (LCA groups henceforth). This is
the natural settings for a general investigation, since the basic
notions used in the formulation of the question -- positive
definiteness, neighborhood of zero, support in and integral over a
$0$-symmetric set $\Omega$ -- can be considered whenever we have
the algebraic and topological structure of an LCA group. Note that
we always have the Haar measure, which makes the consideration of
the integral over a compact set (hence over the support of a
compactly supported positive definite function) well defined. Also
recall that on a LCA group $G$ a function $f$ is called positive
definite if the inequality
\begin{equation}\label{posdefdef}
\sum_{n,m=1}^{N}~ c_n \overline{c_m} f(x_n-x_m)\ge 0 \qquad
(\forall x_1,\dots,x_N\in G, \forall c_1,\dots,c_N\in\CC)
\end{equation}
holds true.  Note that positive definite functions are not assumed
to be continuous. Still, all such functions $f$ are necessarily
bounded by $f(0)$ \cite[p.\ 18, Eqn (3)]{rudin:groups}. Moreover,
$f(x)=\widetilde{f}(x):=\overline{f(-x)}$ for all $x\in G$
\cite[p.\ 18, Eqn (2)]{rudin:groups}, hence the support of $f$ is
necessarily symmetric, and the condition $\supp f \subset \Om$
implies also $\supp f \subset \Om\cap(-\Om)$. The latter set being
symmetric, without loss of generality we can assume at the outset
that $\Om$ is symmetric itself. So in this paper the set $\Omega$
will always be taken to be a $0$-symmetric, open set in $G$.

We find the first mention of the group case in \cite{GRR}, and a
more systematic use of the settings (for the square-integral
Tur\'an problem) in \cite{Domar}. Utilizing also the work in
\cite{hexagon} on extensions to the several dimensional case, the
framework below was set up in \cite{kolountzakis:groups}. There we
obtained some fairly general results for compact LCA groups as
well as for the most classical non-compact groups: $\RR^d$,
$\TT^d$ and $\ZZ^d$.

In this paper we study the problem in the generality of LCA
groups. This simplifies and unifies many of the existing results
and gives several new estimates and examples. If $G$ is a LCA
group a continuous function $f\in L^1(G)$ is positive definite if
its Fourier transform $\ft{f}:\ft{G}\to\CC$ is everywhere
nonnegative on the dual group $\ft{G}$. For the relevant
definitions of the Fourier transform we refer to \cite[Chapter
VII]{katznelson:book} or \cite{rudin:groups}.

We say that $f$ belongs to the class $\FF(\Omega)$ of functions if
$f\in L^1(G)$ is continuous, positive definite and is supported on
a closed subset of $\Omega$. For any positive definite function
$f$ it follows that $f(0) \ge f(x)$ for any $x\in G$. This leads
to the estimate $\int_G f \le \Abs{\Omega} f(0)$ for all
$f\in\FF$, which is called (following Andreev \cite{Andreev}) the
{\em trivial estimate} from now on.

\begin{definition}\label{def:turan-constant}
The {\em Tur\'an constant} ${\cal T}_G(\Omega)$ of a
$0$-symmetric, open subset $\Omega$ of a LCA group $G$ is the
supremum of the quantity $ {\int_G f} /{f(0)} $, where $f\in
\FF(\Omega)$, i.e. $f\in L^1(G)$ is continuous and positive
definite, and $\supp f$ is a closed set contained in $\Omega$.
\end{definition}

In fact, depending on the precise requirements on the functions
considered, here we have certain variants of the problem: an
account of these is presented below in \S \ref{sec:equivalence}.

\begin{remark}
The quantity ${\mathcal T}_G(\Omega)$ depends on which
normalization we use for the Haar measure on $G$. If $G$ is
discrete we use the counting measure and if $G$ is compact and
non-discrete we normalize the measure of $G$ to be 1.
\end{remark}

The {\em trivial upper estimate} or \emph{trivial bound} for the
Tur\'an constant is thus ${\cal T}_G(\Omega) \le \Abs{\Omega}$.

\subsection{Various equivalent forms of the Tur\'an
problem}\label{sec:equivalence}

\noindent In fact, it is worth noting that Tur\'an type problems
can be, and have been considered with various settings, although
the relation of these has not been fully clarified yet. Thus in
extending the investigation to LCA groups or to domains in
Euclidean groups which are not convex, the issue of equivalence
has to be dealt with. One may consider the following function
classes.
\begin{eqnarray}
\label{Flclosed} \FF_1(\Om) &:=& \bigg\{f\in L^1(G)~: ~~\supp f
\subset \Om,~~ f \,{\rm positive \,\,
definite}\, \bigg\}\,, \\
\label{Flcontclosed} \FF_{\&}(\Om) &:=& \bigg\{f\in L^1(G)\cap
C(G)~: ~~\supp f \subset \Om,~~ f \,{\rm positive \,\,
definite}\, \bigg\}\,, \\
\label{Flcompact} \FF_c(\Om) &:=& \bigg\{f\in L^1(G)~: ~~\supp f
\subset \subset \Om,~~ f \,{\rm positive \,\,
definite}\, \bigg\}\,, \\
\label{Fcontcompact} \FF(\Om) &:=& \bigg\{f\in C(G)~: ~~\supp f
\subset \subset \Om,~~ f \,{\rm positive \,\, definite}\,
\bigg\}\,\,.
\end{eqnarray}
In $\FF_1, \FF_{\&}$ $\supp f$ is assumed to be merely closed ad
not necessarily compact, and in $\FF_1, \FF_c$ the function $f$
may be discontinuous.

The respective Tur\'an constants are
\begin{eqnarray}\label{Turanconstants}
\Tu_G^{(1)}(\Om)~ {\rm or} ~ \Tu_G^{\&}(\Om)~ {\rm or} ~
\Tu_G^c(\Om)~ {\rm or} ~\Tu_G(\Om) := \qquad \qquad\qquad\qquad
\qquad\qquad\qquad\\
\qquad \qquad\qquad \sup \bigg\{\frac{\int_G f}{f(0)} \,:~ f \in
\FF_1(\Om) ~ {\rm or} ~ \FF_{\&}(\Om) ~ {\rm or} ~ \FF_c(\Om) ~
{\rm or}~ \FF(\Om), ~ {\rm resp.} \bigg\}.\notag
\end{eqnarray}

In general we should consider functions $f:G\to \CC$. However, it
is easy to see from \eqref{posdefdef} that together with $f$, also
$\overline{f}$ is positive definite. Whence even $\varphi:=\Re f$
is positive definite, while belonging to the same function class.
As we also have $f(0)=\varphi(0)$ and $\int f=\int \varphi$,
restriction to real valued functions does not change the values of
the Tur\'an constants.

For a detailed introduction to positive definite functions, and
for a proof of the following theorem, we refer to
\cite{kolountzakis:groups}.

\begin{theorem}[Kolountzakis-R\'ev\'esz]\label{th:equivalences}
We have for any LCA group the equivalence of the above defined
versions of the Tur\'an constants:
\begin{equation}\label{equivalance:Turanconstants}
\Tu_G^{(1)}(\Om)=\Tu_G^{\&}(\Om)= \Tu_G^c(\Om)=\Tu_G(\Om)\,.
\end{equation}
\end{theorem}

Note that the original formulation, presented also above in
Definition \ref{def:turan-constant}, corresponds to
$\Tu_G^{\&}(\Om)$. Also note that with this setup, e.g. the
interval case $\Omega=[-h,h]\subset \TT$ or $\RR$ admits no
extremal function, because the support of $\Delta_h$ is the full
$\overline{\Omega}$, not a closed subset of the open set $(-h,h)$
In this case an obvious limiting process is neglected in the
formulation of the results above.

\begin{remark} It is not fully clarified what happens for
functions vanishing only outside of $\Om$, but having nonzero
values up to the boundary $\partial\Om$.
\end{remark}

Our main result in this paper appears in Theorem
\ref{th:packingth}. This is an essential extension of the above
mentioned result of Arestov and Berdysheva about convex lattice
tiles in Euclidean spaces being of the Stechkin-Tur\'an type. To
arrive at the result we need some preparations. So in the next
section we describe the structural context, including without
proofs a different extension of the result of Arestov and
Berdysheva - in the direction of spectrality - already given in
\cite{kolountzakis:turan}. Also we explain the relevant new notion
of uniform asymptotic upper density and its computation or
estimation in relation with packing, covering and tiling. The main
result then appears in \S \ref{sec:upper-bound-from-packing}.

\section{Structural properties of sets -- tiling, packing,
spectrality, and uniform asymptotic upper
density}\label{sec:tilingpackingspectra}

\subsection{Tiling and packing}
\label{sec:tiling}

\noindent Suppose $G$ is a LCA group. We say that a nonnegative
function $f\in L^1(G)$ tiles $G$ by translation with a set
$\Lambda\subseteq G$ at level $c\in\CC$ if
$$
\sum_{\lambda\in\Lambda} f(x-\lambda) = c
$$
for a.a. $x\in G$, with the sum converging absolutely. We then
write ``$f+\Lambda = c G$''.

We say that $f$ {\em packs} $G$ with the translation set $\Lambda$
at level $c\in\RR$, and write $f+\Lambda \le c G$, if
$$
\sum_{\lambda\in\Lambda}f(x-\lambda) \le c,
$$
for a.a. $x\in G$.

In particular, a measurable set $\Omega\subseteq\RR^d$ is a {\em
translational tile} if there exists a set $\Lambda\subseteq\RR^d$
such that almost all (Lebesgue) points in $\RR^d$ belong to
exactly one of the translates
$$
\Omega + \lambda,\ \ \ \lambda\in\Lambda.
$$
We denote this condition by $\Omega + \Lambda = \RR^d$.

If $f\in L^1(\RR^d)$ is nonnegative we say that $f$ tiles with
$\Lambda$ at level $\ell$ if
$$
\sum_{\lambda\in\Lambda} f(x-\lambda) = \ell,\ \ \mbox{\rm a.e.\
$x$}.
$$
We denote this latter condition by $f+\Lambda = \ell\RR^d$.


In any tiling the translation set has some properties of density,
which hold uniformly in space. A set $\Lambda\subseteq\RR^d$ has
(uniform) density $\rho$ if
$$
\lim_{R\to\infty} {\#(\Lambda \cap B_R(x)) \over \Abs{B_R(x)}} \to
\rho
$$
uniformly in $x\in\RR^d$. We write $\rho = \dens \Lambda$. We say
that $\Lambda$ has (uniformly) bounded density if the fraction
above is bounded by a constant $\rho$ uniformly for $x\in\RR$ and
$R>1$. We say then that $\Lambda$ has density (uniformly) bounded
by $\rho$.

\begin{remark} It is not hard to prove (see for example \cite{line}, Lemma
2.3, where it is proved in dimension one -- the proof extends
verbatim to higher dimension) that in any tiling $f+\Lambda =
\ell\RR^d$ the set $\Lambda$ has density $\ell / \int f$.
\end{remark}

When the group is finite (and we do not, therefore, have to worry
about the set $\Lambda$ being finite or not) the tiling condition
$f+\Lambda = c G$ means precisely $f*\chi_\Lambda = c$. Taking
Fourier transform, this is the same as $\ft{f}\ft{\chi_\Lambda} =
c\Abs{G}\chi_\Set{0}$, which is in turn equivalent to the
condition
\begin{equation}\label{spectral-condition-for-tiling}
\supp\ft{\chi_\Lambda} \subseteq \Set{0} \cup \Set{\ft{f}=0} \ \
\mbox{and}\ \ c={\Abs{\Lambda} \over \Abs{G}} \sum_{x\in G} f(x).
\end{equation}

Finally, if $E \subseteq G$ we say that $E$ packs with $\Lambda$
if $\chi_E$ packs with $\Lambda$ at level $1$. Observe that $E$
packs with $\Lambda$ if and only if
$$
(E-E) \cap (\Lambda-\Lambda) = \Set{0}.
$$

The packing type condition $\Omega\cap(\Lambda-\Lambda)=\{0\}$
will be used in Theorem \ref{th:packingth} below. This result will
be an essential extension of the earlier result of Arestov and
Berdysheva, stating that in $\RR^d$ a convex lattice tile is
necessary of the Stechkin-Tur\'an type. Another generalization of
this result appears in the next section, through another
structural property of sets, namely spectrality.

\subsection{Spectral sets}
\label{sec:spectral}

\begin{definition}\label{def:spectral} Let $G$ be a LCA group
and $\ft{G}$ be its dual group, that is the group of all
continuous group homomorphisms (characters) $G\to\CC$. We say that
the set $T \subseteq \ft{G}$ is a {\em spectrum} of $H \subseteq
G$ if and only if $T$ forms an orthogonal basis for $L^2(H)$.
\end{definition}


In particular, let $\Omega$ be a measurable subset of $\RR^d$ and
$\Lambda$ be a discrete subset of $\RR^d$. We write $ e_\lambda(x)
= \exp({2\pi i \inner{\lambda}{x}}),\ (x\in\RR^d)$, and $E_\Lambda
= \Set{e_\lambda:\ \lambda\in\Lambda} \subset L^2(\Omega)$. The
inner product and norm on $L^2(\Omega)$ are $\inner{f}{g}_\Omega =
\int_\Omega f \overline{g}, \ \mbox{ and }\ \Norm{f}_\Omega^2 =
\int_\Omega \Abs{f}^2$. The pair $(\Omega, \Lambda)$ is called a
{\em spectral pair} if $E_\Lambda$ is an orthogonal basis for
$L^2(\Omega)$. A set $\Omega$ will be called {\em spectral} if
there is $\Lambda\subset\RR^d$ such that $(\Omega, \Lambda)$ is a
spectral pair. The set $\Lambda$ is then called a {\em spectrum}
of $\Omega$.

\begin{example}If $Q_d = (-1/2, 1/2)^d$ is the cube of unit volume in
$\RR^d$ then $(Q_d, \ZZ^d)$ is a spectral pair, as is well known
by the ordinary $L^2$ theory of multiple Fourier series.
\end{example}

Bent Fuglede formulated the following famous conjecture in 1974.

\begin{conjecture}[Fuglede \cite{fuglede:conjecture}]\label{conj:fuglede}
Let $\Omega \subset \RR^d$ be a bounded open set.  Then $\Omega$
is spectral if and only if there exists $L \subset \RR^d$ such
that $\Omega + L = \RR^d$ is a tiling.
\end{conjecture}

One basis for the conjecture was that the lattice case of this
conjecture is easy to show, (see for example
\cite{fuglede:conjecture,nonsym}). In the following result the
dual lattice $\Lambda^*$ of a lattice $\Lambda$ is defined as
usual by $ \Lambda^* = \Set{x\in\RR^d:\ \forall \lambda\in\Lambda\
\ \inner{x}{\lambda} \in \ZZ}$.

\begin{theorem}[Fuglede \cite{fuglede:conjecture}]\label{th:lattice-fuglede}
The bounded, open domain $\Omega$ admits translational tilings by
a lattice $\Lambda$ if and only if $E_{\Lambda^*}$ is an
orthogonal basis for $L^2(\Omega)$.
\end{theorem}

Note that in Fuglede's Conjecture no relation is claimed between
the translation set $L$ and the spectrum $\Lambda$.

Conjecture \ref{conj:fuglede} in its full generality was recently
disproved. First, T. Tao showed \cite{tao:fuglede-down} that in
$\RR^5$ there exists a spectral set, which however fails to tile
space. The method, roughly speaking, is to construct
counterexamples in finite groups, and then "lift them up" first to
$\ZZ^d$ and finally to $\RR^d$. Soon after that breakthrough,
Tao's construction was further sharpened to provide non-tiling
spectral sets in $\RR^4$ \cite{Mated4} and finally even in
dimension $3$ \cite{MM:hadamard}.

Furthermore, the converse implication was also disproved, first in
dimension 5 by Kolountzakis and Matolcsi
\cite{MM:fuglede-converse}. Subsequently, examples of tiling, but
non-spectral sets were constructed in $\RR^4$ by Farkas and
R\'ev\'esz \cite{FR}, and then even in $\RR^3$ by Farkas, Matolcsi
and M\'ora \cite{FMM}.

Positive results are far more meager, and basically restrict to
special sets on the real line. However, for \emph{planar convex
domains}, it also holds true \cite{IKT}.

As for application of spectrality for estimating the Tur\'an
constant, essentially the following was proved in
\cite{kolountzakis:turan}. (Actually, the possibility of getting
this version from the same proof, appears only in
\cite{kolountzakis:groups}.)

\begin{theorem}[Kolountzakis-R\'ev\'esz]\label{th:spectral-infinite}
If $H$ is a bounded open set in $\RR^d$ which is spectral, then
for the difference set $\Om=H-H$ we have ${\cal T}_{\RR^d}(\Om) =
\Abs{H}$.
\end{theorem}

Originally, we formulated in \cite{kolountzakis:turan} only the
following special case of the above result.
\begin{corollary}[Kolountzakis-R\'ev\'esz]\label{th:spectralRd}
Let $\Omega \subseteq \RR^d$ be a convex domain. If $\Omega$ is
spectral, then it has to be a Stechkin-Tur\'an domain as well.
\end{corollary}

\begin{proof}
First let us note that convex spectral domains are necessarily
symmetric according to the result in \cite{nonsym}. Let now
$\Omega$ be a symmetric convex domain. Then taking $H:=\frac12
\Omega$, we have $H-H=\Omega$. Moreover, if $\Omega$ is spectral,
say with spectrum $\Lambda$, then also $H$ is clearly spectral
with the dilated spectrum $2\Lambda$. So Theorem
\ref{th:spectral-infinite} applies and we are done, in view of
$|H|=|\frac12 \Omega|=|\Omega|/2^d$.
\end{proof}

\begin{corollary}[Arestov-Berdysheva]\label{cor:main}
Suppose the symmetric convex domain $\Omega\subseteq\RR^d$ is a
translational tile. Then it is a Stechkin-Tur\' an domain.
\end{corollary}
\begin{proof}[Proof of Corollary \ref{cor:main}] We start with the
following result which claims that every convex tile is also a
lattice tile.
\begin{theorem}[Venkov \cite{venkov} and McMullen
\cite{mcmullen}]\label{th:venkov} Suppose that a convex body $K$
tiles space by translation. Then it is necessarily a symmetric
polytope and there is a lattice $L$ such that
$$
K+L = \RR^d.
$$
\end{theorem}
A complete characterization of the tiling polytopes is also among
the conclusions of the Venkov-McMullen Theorem but we do not need
it here and choose not to give the full statement as it would
require some more definitions.

So, if a convex domain is a tile, it is also a lattice tile, hence
spectral by Theorem \ref{th:lattice-fuglede}, and as such it is
Stechkin-Tur\'an, by Corollary \ref{th:spectralRd}. \end{proof}

\begin{remark} If one wants to avoid using the
Venkov-McMullen theorem in the proof of Corollary \ref{cor:main}
one should enhance the assumption of Corollary \ref{cor:main} to
state that $\Omega$ is a lattice tile. Arestov and Berdysheva in
\cite{arestov:tiles} prove Corollary \ref{cor:main} without going
through spectral domains.
\end{remark}

The result of \cite{hexagon} about the hexagon being a
Stechkin-Tur\' an domain is thus a special case of our Corollary
\ref{cor:main}, but not the result in \cite{Siegel} and
\cite{gorbachev:ball} about the ball being Stechkin-Tur\'an type.
The ball, and essentially every smooth convex body
\cite{IKT:smooth}, is known not to be spectral, in accordance with
the Fuglede Conjecture.

\subsection{The notion of uniform asymptotic upper density on LCA groups
}\label{sec:density}

\noindent First let us recall the frequently used definition of
asymptotic uniform upper density in $\RR^d$. Let $K\subset\RR^d$
be a \emph{fat body}, i.e. a set with $0\in {\rm int} K$,
$K=\overline{\rm int K}$ and $K$ compact. Then uniform asymptotic
upper density in $\RR^d$ with respect to $K$ is defined as
\begin{equation}\label{RUdensity}
\overline{D}_{K}(A) := \limsup_{r\to\infty} \frac{\sup_{x\in\RR^d}
|A\cap (rK+x)|}{|rK|}~.
\end{equation}
It is obvious that the notion is translation invariant. It is also
well-known, that $\overline{D}_{K}(A)$ gives the same value for
all nice - e.g. for all convex - bodies $K\subset \RR^d$, although
this fact does not seem immediate from the formulation.

Note also the following ambiguity in the use of densities in
literature. Sometimes even in continuous groups a discrete set
$\Lambda$ is considered in place of $A$, and then the definition
of the asymptotic upper density is
\begin{equation}\label{RUNdensity}
\overline{D}^{\#}_{K}(A) := \limsup_{r\to\infty}
\frac{\sup_{x\in\RR^d} {\rm \#} (\Lambda \cap (rK+x))}{|rK|}~.
\end{equation}

That motivates the general definition of asymptotic uniform upper
densities of {\em measures}, say measure $\nu$ with respect to
measure $\mu$, whether equal or not. E.g. in \eqref{RUNdensity}
$\nu:={\rm \#}$ is the cardinality or counting measure, while
$\mu:=|\cdot|$ is just the volume. The general formulation in
$\RR^d$ is thus
\begin{equation}\label{Rnu-density}
\overline{D}_{K}(\nu):= \limsup_{r\to\infty}
\frac{\sup_{x\in\RR^d}\nu(rK+x)}{|rK|}~.
\end{equation}

Two notions of asymptotic uniform upper densities of measures
$\nu$ with respect to a translation invariant, nonnegative,
locally finite (outer) measure $\mu$ were defined in general LCA
groups in \cite{R}. Considering such groups are natural for they
have an essentially unique translation invariant Haar measure
$\mu_G$ (see e.g. \cite{rudin:groups}), what we fix to be our
$\mu$. By construction, $\mu$ is a Borel measure, and the sigma
algebra of $\mu$-measurable sets is just the sigma algebra of
Borel mesurable sets, denoted by $\BB$ throughout. To avoid
questions of infinite measure, we consider the subset $\BB_0$ of
Borel measurable sets having compact closure.

Note if we consider the discrete topological structure on any
abelian group $G$, it makes $G$ a LCA group with Haar measure
$\mu_G={\rm \#}$, the counting measure. This is the natural
structure for $\ZZ^d$, e.g. On the other hand all $\sigma$-finite
groups admit the same structure as well, unifying considerations.
(Note that e.g. $\ZZ^d$ is not a $\sigma$-finite group since it is
{\em torsion-free}, i.e. has no finite subgroups.)

The other measure $\nu$ can be defined, e.g., as the {\em trace}
of $\mu$ on the given set $A$, that is,
$\nu(H):=\nu_A(H):=\mu_G(H\cap A)$, or can be taken as the
counting measure of the points included in some set $\Lambda$
derived from the cardinality measure similarly:
$\gamma(H):=\gamma_{\Lambda}(H):={\rm \#} (H\cap \Lambda)$.

\begin{definition}\label{compactdensity}
Let $G$ be a LCA group and $\mu:=\mu_G$ be its Haar measure. If
$\nu$ is another measure on $G$ with the sigma algebra of
measurable sets being ${\mathcal S}$, then we define
\begin{equation}\label{Cnudensity}
\overline{D}(\nu;\mu) := \inf_{C\Subset G} \sup_{V\in {\mathcal S}
\cap {\mathcal B_0}} \frac{\nu(V)}{\mu(C+V)}~.
\end{equation}
In particular, if $A\subset G$ is Borel measurable and $\nu=\mu_A$
is the trace of the Haar measure on the set $A$, then we get
\begin{equation}\label{CAdensity}
\overline{D}(A) :=\overline{D}(\nu_A;\mu) := \inf_{C\Subset G}
\sup_{V\in \mathcal B_0} \frac{\mu(A\cap V)}{\mu(C+V)}~.
\end{equation}
If $\Lambda\subset G$ is any (e.g. discrete) set and $\gamma
:=\gamma_\Lambda:=\sum_{\lambda\in\Lambda} \delta_{\lambda}$ is
the counting measure of $\Lambda$, then we get
\begin{equation}\label{CLdensity}
\overline{D}^{\#}(\Lambda) :=\overline{D}(\gamma_{\Lambda};\mu) :=
\inf_{C\Subset G} \sup_{V\in {\mathcal B_0}} \frac{{\rm \#
}(\Lambda\cap V)}{\mu(C+V)}~.
\end{equation}
\end{definition}

\begin{proposition}\label{prop:densequivalence}
Let $K$ be any convex body in $\RR^d$ and normalize the Haar
measure of $\RR^d$ to be equal to the volume $|\cdot|$. Let $\nu$
be any measure with sigma algebra of measurable sets $\mathcal S$.
Then we have
\begin{equation}\label{Rd-equivalance}
\overline{D}(\nu;|\cdot|) = \overline{D}_K(\nu)~.
\end{equation}
\end{proposition}

The same statement applies also to $\ZZ^d$. For heuristical
considerations and comparisons to existing notions and approaches,
as well as for the proofs and for some examples we refer to
\cite{R}.

\subsection{Packing, covering, tiling and uniform asymptotic upper
density}\label{sec:packcoverdens}

\begin{proposition}\label{prop:packdens} Assume that $H\in
\BB$ and that $H+\Lambda\leq G$ ($H$ packs $G$ with
$\Lambda\subset G$), i.e. $(H-H)\cap(\Lambda-\Lambda)\subseteq
\{0\}$. Then $\Lambda$ must satisfy
$\overline{D}^{\#}(\Lambda)\leq 1/\mu(H)$.
\end{proposition}
\begin{proof} Let $B\Subset H$ and $V\in \BB_0$ be arbitrary.
Denote $L:=\Lambda\cap V$. Then $B+V\supset B+L=\cup_{\lambda\in
L} (B+\lambda)$, and this union being disjoint (as
$(B+\lambda)\cap (B+\lambda')\subset (H+\lambda)\cap (H+\lambda')
=\emptyset$ unless $\lambda=\lambda'$), from additivity and
translation invariance of the Haar measure we obtain $\mu(B+V)\geq
\mu(B+L)= \# L \mu(B)$. This yields $\# L/ \mu(B+V) \leq
1/\mu(B)$, therefore $\sup_{V\in\BB_0} \#(\Lambda \cap V)/
\mu(B+V) \leq 1/\mu(B)$. Approximating $\mu(H)$ by $\mu(B)$ of
$B\Subset H$ arbitrarily closely, we thus obtain $\inf_{B\Subset
H} \sup_{V\in\BB_0}\#(\Lambda \cap V)/ \mu(B+V) \leq 1/\mu(H)$.
However, $\overline{D}^{\#}(\Lambda)$ is a similar infimum
extended to a larger family of compact sets, so it can not be
larger, and the assertion follows.
\end{proof}

\begin{proposition}\label{prop:coverdens} Assume that $H\in\BB_0$
and that it covers $G$ with $\Lambda\subset G$ ("$H+\Lambda\geq
G$"), i.e. $H+\Lambda$ contains $\mu$-almost all points of $G$.
Then we necessarily have $\overline{D}^{\#}(\Lambda)\geq
1/\mu(H)$.
\end{proposition}
\begin{proof} Let $C\Subset G$ be arbitrary, and take
$W:=\overline{H}-C$, which is again a compact set of $G$ by
assumption on $H$ and in view of the continuity of the group
operation on $G$. So the Theorem in \S 2.6.7. on p. 52 of
\cite{rudin:groups} applies to the compact set $W$ and to any
given $\ve>0$, and we find some Borel measurable set $U=U_{\ve,
C}\in \BB_0$ satisfying $\mu(U-W)<(1+\ve)\mu(U)$.

Consider now $V:=V_{\ve,C}:=U-H\in \BB_0$. Then
$\mu(C+V)=\mu(C+U-H)\leq \mu(U-(\overline{H}-C))= \mu(U-W)<
(1+\ve)\mu(U)$. Denote $L:=\Lambda\cap V$. Then
$L=\{\lambda\in\Lambda~:~ \exists h\in H,\ \lambda+h\in
U\}=\{\lambda\in \Lambda~:~(\lambda+H)\cap U \ne \emptyset\}$, and
so clearly $U\cap(\Lambda+H) = \cup_{\lambda\in L}(\lambda+H)$,
while $U_0:=U\setminus (U\cap(\Lambda+H))$ is of measure zero by
assumption on the covering property of $H$ with $\Lambda$. So in
all $\mu(U)\leq  \mu(U_0) + \sum_{\lambda\in L}\mu(\lambda+H)
=0+\# L \mu(H)$ and $\mu(C+V)<(1+\ve)\# L\mu(H)$.

It follows that with the arbitrarily chosen $C\Subset G$ we have
with a certain $V_{\ve,C}\in \BB_0$
$$
\frac{\#(\Lambda\cap V_{\ve,C})}{\mu(C+V_{\ve,C})} \geq
\frac{1}{(1+\ve)\mu(H)},
$$
so taking supremum over all $V\in\BB_0$ we even get
$\sup_{V\in\BB_0} \#(\Lambda\cap V)/\mu(C+V) \geq 1/\mu(H)$. This
holding for all $C\Subset G$, taking infimum over $C$ does not
change the lower estimation, so finally we arrive at
$\overline{D}^{\#}(\Lambda)\geq 1/\mu(H)$, whence the proposition.
\end{proof}

Because tiling means simultaneously packing and covering.
Therefore, from the above two propositions the following corollary
obtains immediately.

\begin{corollary}\label{cor:tilingdens} Assume that $H\in \BB_0$
tiles with the set of translations $\Lambda\subset G$:
$H+\Lambda=G$. Then we also have $\overline{D}^{\#}(\Lambda)=
1/\mu(H)$.
\end{corollary}


\section{Upper bound from packing}
\label{sec:upper-bound-from-packing}


\subsection{Bounds from packing in some special cases}

In the type of results we now present, some kind of ``packing''
condition is assumed on $\Omega$ which leads to an upper bound for
${\cal T}_G(\Omega)$. The first result we present here is taken
from \cite{kolountzakis:groups}: we repeat it here for sake of a
simpler situation which nevertheless may shed light on the general
case.

\begin{theorem}[Kolountzakis-R\'ev\'esz]\label{th:diff}
Suppose that $G$ is a compact abelian group, $\Lambda\subseteq G$,
$\Omega\subseteq G$ is a $0$-symmetric open set and
$(\Lambda-\Lambda) \cap \Omega \subseteq \Set{0}$. Suppose also
that $f\in L^1(G)$ is a continuous positive definite function
supported on $\Omega$. Then
\begin{equation}\label{bound-from-diff}
\int_G f(x)\,dx \le {\mu(G) \over \# {\Lambda}} f(0).
\end{equation}
In other words ${\cal T}_G(\Omega) \le \mu(G) / \#{\Lambda}$.
\end{theorem}
(Observe that the conditions imply that $\Lambda$ is finite.)

\begin{proof} Define $F:G\to\CC$ by
$$
F(x) = \sum_{\lambda,\mu\in\Lambda}f(x+\lambda-\mu).
$$
In other words $F = f*\delta_\Lambda*\delta_{-\Lambda}$, where
$\delta_A$ denotes the finite measure on $G$ that assigns a unit
mass to each point of the finite set $A$. It follows that $\ft{F}
= \ft{f}\Abs{\ft{\delta_\Lambda}}^2 \ge 0$ so that $F$ is
continuous and positive definite. Moreover, we also have
\begin{equation}
\supp F \subseteq \supp f + (\Lambda-\Lambda)
  \subseteq \Omega+(\Lambda-\Lambda)
\end{equation}
and
\begin{equation}\label{Ff-0}
F(0) = \#  \Lambda f(0),
\end{equation}
since $\Omega \cap (\Lambda-\Lambda) \subseteq \Set{0}$. Finally
\begin{equation}\label{Ff-1}
\int_G F = \#  \Lambda^2 \int_G f.
\end{equation}
Applying the trivial upper bound $\int_G F \le F(0)
\mu(\Omega+(\Lambda-\Lambda))$ to the positive definite function
$F$ and using \eref{Ff-0} and \eref{Ff-1} we get
\begin{equation}\label{bound}
\int_G f \le {\mu(\Omega+(\Lambda-\Lambda)) \over \#  \Lambda}
f(0).
\end{equation}
Estimating trivially $\mu(\Omega+(\Lambda-\Lambda))$ from above by
$\mu(G)$ we obtain the required ${\cal T}_G(\Omega) \le \mu(G) /
\#
 \Lambda$.
\end{proof}

\begin{corollary}\label{cor:packing}
Let $G$ be a compact abelian group and suppose $\Omega,H,\Lambda
\subseteq G$, $H+\Lambda\le G$ is a packing at level $1$, that
$\Omega \subseteq H-H$ and that $f\in\FF(\Om)$. Then
\eqref{bound-from-diff} holds.

In particular, if $H+\Lambda=G$ is a tiling, we have
\begin{equation}\label{bound-for-tiles}
{\cal T}_G(\Omega) \le \mu(H).
\end{equation}
\end{corollary}

\begin{proof}
Since $H+\Lambda \le G$ it follows that $(H-H) \cap
(\Lambda-\Lambda) = \Set{0}$. Since $\Omega \subseteq H-H$ by
assumption it follows that $\Omega$ and $\Lambda-\Lambda$ have at
most $0$ in common. Theorem \ref{th:diff} therefore applies and
gives the result. If $H+\Lambda=G$ then $\mu(G)/\#
\Lambda=\Abs{H}$ and this proves \eref{bound-for-tiles}.
\end{proof}

A partial extension of the result to the non-compact case was also
worked out in \cite{kolountzakis:groups}. However, it used the
notion of u.a.u.d. which then restricted considerations to
classical groups only.

\begin{theorem}[Kolountzakis-R\'ev\'esz]\label{th:diff-infinite}
Suppose that $G$ is one of the groups $\RR^d$ or $\ZZ^d$, that
$\Lambda \subseteq G$ is a set of uniform asymptotic upper density
$\rho>0$, and $\Omega\subseteq G$ is a $0$-symmetric open set such
that $\Omega \cap (\Lambda-\Lambda) \subseteq \Set{0}$. Let also
$f\in L^1(G)$ be a continuous positive definite function on $G$
whose support is a compact set contained in $\Omega$. Then
\begin{equation}\label{bound-from-diff-infinite}
\int_G f(x)\,dx \le {1 \over \rho} f(0).
\end{equation}
In other words ${\cal T}_G(\Omega) \le 1/\rho$.
\end{theorem}

For sharpness and examples we refer to \cite{kolountzakis:groups}.
Note that some parts of the proof in \cite{kolountzakis:groups}
for this theorem will be used even in the proof for our more
general result, see the end of Lemma \ref{l:TGstandard}.

\subsection{Bounds from packing in general LCA groups}

Now we have ready a notion of u.a.u.d. as defined in \S
\ref{sec:density}. With this notion, we have the following general
version of the above particular results.

\begin{theorem}\label{th:packingth} Let $\Omega\subset G$ be a
$0$-symmetric open neighborhood of $0$ and $\Lambda\subset G$ be a
subset satisfying the ''packing-type condition"
$\Omega\cap(\Lambda-\Lambda)=\{0\}$. If
$\rho:=\overline{D}^{\#}(\Lambda)>0$, then we have $\Tu_G(\Omega)
\le 1/\rho$.
\end{theorem}

\begin{proof}
Let $\e>0$ be fixed small, but arbitrary. By Theorem
\ref{th:equivalences}, there exists $f\in\FF(\Om)$ with $\int_G f
> \Tu_G(\Om)-\e$. Denote $S:=\supp f$, which is a compact subset
of $\Om$ in view of $f\in\FF(\Om)$.

In the following we consider a compact, $0$-symmetric neighborhood
of $0$ which we denote by $W$. We require $W$ to be the closure of
a $0$-symmetric open subset $O$ containing $S-S$ in it. (Such a
compact set exists: by continuity of the group operation, the
compact subset $S\times S$ is mapped to a compact set, i.e. $S-S$
is compact, and then for any symmetric, open neighborhood Q of $0$
with compact closure $\overline{Q}$ choosing $O:=(S-S)+Q$,
$W:=(S-S)+\overline{Q}$ suffices.)

Let us consider the subgroup $G_0$ of $G$, generated by $W$. Here
we repeat the construction on \cite[p. 52]{rudin:groups}. First,
by \cite[Lemma 2.4.2]{rudin:groups}, $\langle W \rangle =G_0$
implies that there exists a closed subgroup $K\le G_0$ which is
isomorphic to $\ZZ^k$ with some natural number $k$ and satisfies
$W\cap K=\{0\}$, so that $H:=G_0/K$ is then compact. Let $\phi$ be
the natural homomorphism (projection) of $G_0$ onto $H$.

Because $S-S\subset {\rm int} W$, there exists an open
neighborhood $X_1$ of $S$ such that $X_1-X_1\subset W$, whence
$\phi(x)-\phi(y)=0\in H$ with $x,y\in X_1$ would imply $x-y\in
{\rm ker} \phi = K$, i.e. $x-y \in K\cap W = \{0\}$ and thus
$x=y$. In other words, $\phi$ is a homeomorphism on $X_1$, and
$Y_1:=\phi(X_1)\subset H$ is open. By compactness of $H$, finitely
many translates of $Y_1$, say $Y_1,Y_2,\dots,Y_r$ will cover $H$,
and there are open subsets $X_i$ of $G_0$ with compact closure
such that $\phi$ maps $X_i$ onto $Y_i$ homeomorphically for each
$i=1,\dots,r$. If $Y'_{1}:=Y_{1}$, $Y'_{i}:=Y_{i}\setminus
(\cup_{j=1}^{i-1} Y_j)$ ($i=2,\dots,r$) and
$X_i':=X_i\cap\phi^{-1}(Y_i')$ ($i=1,\dots,r$), then
$E:=\cup_{i=1}^r X'_i$ is a Borel set in $G_0$ with compact
closure, $\phi$ is one-to-one on $E$, and $\phi(E)=H$, i.e., each
$x\in G_0$ can be uniquely represented as $x=e+n$, with $e\in E$
and $n\in K$.

In the following we put $\|n\|:=\max_{1\le j \le k} |n_j|$, where
$(n_1,\dots,n_k)\in \ZZ^k$ is the element corresponding to $n\in
K$ under the fixed isomorphism from $K$ to $\ZZ^k$. Note also that
$S\subset X_1=X_1'\subset E$ and that $\overline{E}$ is compact.
Hence also $E+E-E-E$ has compact closure, and the discrete set $K$
can intersect it only in finitely many points. So we put
$s:=\max\{\|n\|~:~ n\in (E+E-E-E)\cap K\}$, which is finite. Next
we define
\begin{equation}\label{VNdef}
V_N:=\cup \left\{E+n~:~n\in K, \|n\|\le N \right\}\qquad
(N\in\NN).
\end{equation}
Note that $|V_N|=(2N+1)^k|E|$ for all $N\in\NN$, and the $V_N$ are
Borel sets with compact closure. Let $N,M\in \NN$, and $x=e+n$,
$y=f+m$ be the decomposition of two elements $x\in V_N$ and $y\in
V_M$ in terms of $E+K$, i.e. $e,f\in E$ and $n,m\in K$. Then
$x+y=e+f+n+m=g+p+n+m$, where $e+f$ has the standard decomposition
$g+p$, and so $p=e+f-g\in (E+E-E)$, therefore in $(E+E-E)\cap K$,
and we find $\|p\|\le s$. In all, we find $x+y\in E+q$, where
$q:=p+n+m$ satisfies $q\le N+M+s$, and so $x+y\in V_{N+M+s}$. It
follows that $V_N+V_M\subset V_{N+M+s}$.

\begin{lemma}\label{l:TGstandard}
With the above notations we have $\Tu_{G_0}(V_N)\le (N+2+s)^k|E|$
for arbitrary $N\in\NN$.
\end{lemma}
\begin{proof} Consider again the natural homeomorphism (projection)
$\phi:G_0\rightarrow G_0/K=:H$. \cite[Proposition
3]{kolountzakis:groups} gives
\begin{equation}\label{propositionyield}
\Tu_{G_0}(V_N)\le C \Tu_H(\phi(V_N))\Tu_K(V_N\cap K)\qquad
(C:=\frac{d\nu}{d\mu_H})
\end{equation}
with $\nu:=\mu_{G_0/K}\circ\pi\circ\phi^{-1}=\mu_{G_0/K}$, as
$\pi=\phi$ in our case. Note that now $G_0/K:=H$, but the Haar
measures are normalized differently: $H$, as a compact group, has
$\mu_H(H)=1$, $K\cong\ZZ^k$ has the counting measure as its
natural Haar measure, but $G_0$ has the restriction measure
$\mu_{G_0}$ inherited from $|\cdot|=\mu_G$. Therefore, following
the standard convention (as explained e.g. in \cite[\S
2.7.3]{rudin:groups}), under what convention  the above quoted
\cite[Proposition 3]{kolountzakis:groups} holds, we must take care
of $d\mu_{G_0}=d\mu_{G_0/K}d\mu_K$, which determines
$d\mu_{G_0/K}$ and hence $C$. It suffices to consider one test
function, which we chose to be $\chi_E$, the characteristic
function of $E$. We obtain
\begin{align}
|E|=\mu_{G_0}(E)&=\int_{G_0}\chi_E d\mu_{G_0} =
\int_{G_0/K}\int_K\chi_E(x+y) d\mu_K(y)d\mu_{G_0/K}([x]) \notag \\
&=\int_{G_0/K} 1 ~~d\mu_{G_0/K}([x])= \mu_{G_0/K}(G_0/K)
\end{align}
in view of $\# \{y\in K~:~ x+y \in E\}=1$ by the above unique
representation of $G_0$ as $E+K$. It follows that
\begin{equation}\label{Cvalue}
C\left(:=\frac{d\nu}{d\mu_H}\right)=\frac{\mu_{G_0/K}(G_0/K)}{\mu_H(H)}=|E|
\end{equation}
and we are led to
\begin{equation}\label{TuVN estimate}
\Tu_{G_0}(V_N)\le |E| \Tu_H(\phi(V_N))\Tu_K(V_N\cap K).
\end{equation}
Since $E\subset V_N$ and $\phi(E)=H$, $\Tu_H(\phi(V_N)=
\Tu_H(H)=1$. Let us write from now on  $Q_M:=\{m~:~ \|m\|\le M\}$.
On the other hand $V_N\cap K \subset Q_{N+s}$, because for any
$e\in E\cap K$ we necessarily have $\|e\|\le s$. These
observations yield
$$
\Tu_{G_0}(V_N)\leq |E| \cdot 1\cdot \Tu_K\left(\{m\in K~:~
\|m\|\leq N +s\}\right)=|E| \Tu_{\ZZ^k}(Q_{N+s}),
$$
by the isomorphism of $K$ and $\ZZ^k$. It remains to see that
$\Tu_{\ZZ^k}(Q_L)\leq (L+2)^k$, which follows from \cite[formula
(26)]{kolountzakis:groups} from the proof of Theorem
\ref{th:diff-infinite} in \cite{kolountzakis:groups}.
\end{proof}

\begin{lemma}\label{l:translationlemma} Let $V$ be any Borel measurable
subset of $G$ with compact closure and let $\nu$ be a Borel
measure on $G$ with $\overline{D}_{G}(\nu;\mu)=\rho> 0$. If $\e>0$
is given, then there exists $z\in G$ such that
\begin{equation}\label{VN+z}
\nu(V+z)\geq (\rho-\e)|V|.
\end{equation}
\end{lemma}

\begin{proof} Let $D:=-V$. $V$ is a Borel set with compact closure
$\overline{D}\Subset G$. So by Definition \ref{compactdensity} we
can find, according to the assumption on
$\overline{D}_{G}(\nu;\mu)=\rho$, some $Z\in \BB_0$ which satisfy
\begin{equation}\label{nuxZ}
\nu(Z)\geq (\rho-\e)|Z+\overline{D}|\geq (\rho-\e)|Z+D|.
\end{equation}
We can then write
\begin{equation}\label{intchiZt}
\int \chi_Z(t) d\nu(t) \geq (\rho-\e) |Z+D| .
\end{equation}
For $t\in Z$ $u\in D(=-V)$ also $t+u\in Z+D$, hence
$\chi_{Z+D}(t+u)=1$, and we get
\begin{equation}\label{chiZtesti}
\chi_Z(t) \leq \frac{1}{D} \int \chi_{Z+D} (t+u)\chi_D(u) d\mu(u)
\end{equation}
for all $t\in Z$. But for $t\not\in Z$ $\chi_Z(t)=0$ and the right
hand side being nonnegative, inequality \eqref{chiZtesti} holds
for all $t\in G$, hence \eqref{intchiZt} implies
\begin{align}
(\rho-\e)|Z+D| & \leq \frac1{|D|} \int\int \chi_{Z+D} (t+u)
\chi_D(u)
d\mu(u) d\nu(t) \notag \\
&=\int \chi_{Z+D} (y) \left( \frac1{|D|} \int\chi_D(y-t) d\nu(t)
\right) d\mu(y) \notag \\
&= \int \chi_{Z+D} (y) f(y) d\mu(y) \qquad\qquad \left(
\textrm{with}\quad f(y):= \frac{\nu(y-D)}{|D|} \right) \\ & =
\int_{Z+D} f d\mu .\notag
\end{align}
It follows that there exists $z\in Z+D\subset G$ satisfying $f(z)
\geq (\rho-\e)$. That is, we find $\nu(z-D) \geq (\rho-\e) |D|$ or
$\nu(z+V)=\nu(z-D)\geq (\rho-\e) |D| =(\rho-\e)|V|.$
\end{proof}

\begin{lemma}\label{l:subgrouplemma} If $\overline{D}_G(\nu;\mu)=\rho>0$
with $\mu=\mu_G$ and $\nu$ any given Borel measure on the LCA
group $G$, then for any open subgroup $G'$ of $G$, compact
$D\Subset G'$ and $\e>0$ there exist $x\in G$ and $Z\subset G'$,
$Z\in\BB_0$ so that $\nu(Z+x)\geq (\rho-\e)\mu(Z+D)$.
\end{lemma}

\begin{remark} One would be tempted to assert that on some coset
$G'+x$ of $G'$ the relative density of $\nu$ must be at least
$\rho-\e$, i.e. $\overline{D}_{G'}(\nu_x;\mu|_{G'})=\rho-\e$ with
$\nu_x(Z):=\nu(Z+x)$ for $Z\subset G'$ Borel and $x\in G$.
However, this stronger statement does not hold true. Consider e.g.
$G=\ZZ^2$, $G':=\ZZ\times\{0\}$, $A:=\{(k,l)~:~ k\in \NN, l\geq k
\}$, and $\nu:=\mu_A$ the trace of the counting measure $\mu$ of
$\ZZ^2$ on $A$. Since $A$ contains arbitrarily large squares,
$\overline{D}(\nu;\mu)=1$. (In fact, $\nu$ has a positive
asymptotic density $\delta(\nu;\mu)=1/8$, too.) However, for each
coset $G'+x=\ZZ\times\{m\}$ of $G'$ the intersection $A\cap G'$ is
only finite and $\overline{D}_{G'}(\nu_x;\mu|_{G'})=0$.
\end{remark}

\begin{proof} By condition, for $D\Subset G' \leq G$ there exists
$V\Subset G$ such that
\begin{equation}\label{muVmuVD}
\nu(V)\geq (\rho-\e)\mu(V+D).
\end{equation}
Let now $U$ be an open set containing $V+D$ and with compact
closure $\overline{U}\Subset G$. Because the cosets of $G'$ cover
$G$, we have
$$
V+D=\bigcup_{x\in G} \left( (V+D)\cap(G'+x)\right) \subset
\bigcup_{x\in G} \left( U \cap(G'+x) \right).
$$
Since both $U$ and $G'$ are open, and $V+D$ is compact, the
covering on the right hand side has a finite subcovering;
moreover, we can select all covering cosets only once, hence
arrive at a disjoint covering
$$
V+D \subset \bigcup_{j=1}^{m} U_j\qquad\qquad \left( ~U_j:=U
\cap(G'+x_j), \quad j=1,\dots,m ~\right).
$$
Take now $V_j:=U_j\cap(V+D)$. As the $U_j$ are disjoint, so are
the $V_j$; and as the $U_j$ together cover $V+D$, so do the $V_j$.
So we have the disjoint covering $V+D=\cup_{j=1}^m V_j$.
Furthermore, if $x\in (V+D)\cap(G'+x_j) \subset V+D$, it must
belong to $V_j$, for all $V_i$ with $i\ne j$ are disjoint from
$G'+x_j$ and hence $x\not\in V_i$ for $i\ne j$. Therefore all
$V_j$ are compact, in view of $V_j=U_j\cap(V+D)=(V+D)\cap
U\cap(G'+x_j)=(V+D)\cap(G'+x_j)$ because $V+D$ is compact and
$G'+x_j$ is also closed (as an open subgroup, hence its cosets,
are always closed, too.)

Next we define $W_j:=V\cap V_j$. Plainly, $W_j\Subset G$ and
disjoint, and $V=\cup_{j=1}^m W_j$. Moreover, $W_j+D=V_j$; indeed,
$W_j+D=(V\cap(G'+x_j))+D=(V+D)\cap(G'+x_j)$ since $D\subset G'$
and $G'\leq G$. So we find
\begin{equation}\label{nuVnuW}
\nu(V)=\sum_{j=1}^m \nu(W_j)
\end{equation}
and also
\begin{equation}\label{muVD}
\mu(V+D)=\sum_{j=1}^m \mu(V_j)= \sum_{j=1}^m \mu(W_j+D)=
\sum_{j=1}^m \mu(W_j-x_j+D)
\end{equation}
Collecting \eqref{nuVnuW}, \eqref{muVmuVD} and \eqref{muVD} we
conclude
\begin{equation}\label{nuW}
\sum_{j=1}^m \nu(W_j) \geq (\rho-\e) \sum_{j=1}^m \mu(W_j-x_j+D),
\end{equation}
hence for some appropriate $j\in [1,m]$ we also have $
\nu(W_j)\geq (\rho-\e) \mu (W_j-x_j+D) $. Taking $Z:=W_j-x_j$ and
$x=x_j$ concludes the proof.
\end{proof}

\noindent \emph{End of the proof of Theorem \ref{th:packingth}.}
Let now $\nu:=\delta_{\Lambda}$ be the counting measure of the
(discrete) set $\Lambda\subset G$. Then
$\overline{D}_G(\nu;\mu)=\overline{D}_G^{\textrm{\#}}(\Lambda)=\rho>0$
and Lemma \ref{l:translationlemma} applies providing some
$z:=z_N\in G$ with
\begin{equation}\label{MVN}
M:=\textrm{\#} \left(\Lambda\cap(V_N+z)\right) \geq
(\rho-\e)|V_N|.
\end{equation}
Take now
$\Lambda':=\Lambda\cap(V_N+z)=\{\lambda_m~:~m=1,\dots,M\}$. Put
$F:=f\star \delta_{\Lambda'} \star \delta_{-\Lambda'}$, i.e.
$$
F(x):=\sum_{m=1}^M \sum_{n=1}^M f(x+\lambda_m-\lambda_n),
$$
which is a positive definite continuous function supported in
$S+(V_N+z)-(V_N+z)=S+V_N-V_N = S+E-E+Q_{2N}\subset
E+E-E+Q_{2N}\subset V_{2N+s}$. Furthermore, as $S\subset G_0$,
\begin{equation}\label{intF}
\int_{G_0} F = M^2 \int_{G_0} f \geq M^2 (\Tu_G(\Omega)-\e)
\end{equation}
and
\begin{equation}\label{Fnull}
F(0)=\sum_{m=1}^M \sum_{n=1}^M f(\lambda_m-\lambda_n)=Mf(0)=M,
\end{equation}
because if $\lambda_m-\lambda_n\in S$ then $\lambda_m-\lambda_n\in
S\cap(\Lambda-\Lambda)\subset \Omega\cap (\Lambda-\Lambda)=\{0\}$
and $\lambda_m=\lambda_n$, i.e. $n=m$. By this construction we
derive that
\begin{align}\label{TVbelow}
\Tu_{G_0}(V_{2N+s}) & \geq \frac{1}{F(0)} \int_{G_0} F \geq M
(\Tu_G(\Omega)-\e)  \notag \\ &\geq (\rho-\e) (\Tu_G(\Omega)-\e)
|V_N| = (\rho-\e)(\Tu_G(\Omega)-\e) (2N+1)^k|E|.
\end{align}
On the other hand Lemma \ref{l:TGstandard} provides us
\begin{equation}\label{TVabove}
\Tu_{G_0}(V_{2N+s}) \leq (2N+s+2)^k |E|.
\end{equation}
On comparing \eqref{TVbelow} and \eqref{TVabove} we conclude
$(\rho-\e)(\Tu_G(\Omega)-\e) (2N+1)^k|E| \leq (2N+s+2)^k|E| $,
that is
$$
\Tu_G(\Omega)-\e \leq \frac{1}{\rho-\e} \left(\frac{2N+s+2}{2N+1}
\right)^k.
$$
Letting $N\to\infty$ and $\e\to 0$ gives the assertion.
\end{proof}

\begin{corollary} Suppose that $\Omega\subset G$ is an open
and symmetric set and $\Omega=H-H$, where $H$ tiles space with
$\Lambda\subset G$. Moreover, assume that $H$ has compact closure
$\overline{H}\Subset G$ and is measurable, i.e. $H\in \BB_0$. Then
$\Tu_G(\Om)= \mu(H)$.
\end{corollary}
\begin{proof} First, observe that for any $A\Subset H$ we have
$f:=\chi_A*\chi_{-A} \in \FF_{\&}(\Om)$. Indeed,
$\widetilde{\chi_A}=\chi_{-A}$ because $\chi_A$ is real valued,
also $\chi_A\in L^2(G)$, and such a convolution representation
guarantees that $f\in C(G)\cap L^1(G)$ is positive definite;
furthermore, if $f(x)\ne 0$, then necessarily $x=a-a'$ with some
$a,a'\in A\subset H$, hence $\supp f \subset \Omega$.

Therefore, calculating with the admissible function $f$, we find
$\Tu_G(\Omega) \geq \int_G f/f(0) = \mu(A)^2/\mu(A) =\mu(A)$.
Since $H$ is Borel measurable, its measure can be approximated
arbitrarily closely by measures of inscribed compact sets $A$:
therefore, taking supremum over compact sets $A\Subset H$, we
obtain the lower estimate $\Tu_G(\Om)\geq \mu(H)$.

On the other hand, $H+\Lambda=G$ entails that $H$ packs with
$\Lambda$, and so an application of Theorem \ref{th:packingth}
gives $\Tu_G(\Om)\leq 1/\overline{D}^{\#}(\Lambda)$. Now we can
apply that $H$ also covers $G$ with $\Lambda$, so that Proposition
\ref{prop:coverdens} also applies, giving
$\overline{D}^{\#}(\Lambda)\geq 1/\mu(H)$. On combining the last
two inequalities, $\Tu_G(\Om)\leq \mu(H)$, whence the assertion,
follows.
\end{proof}


\noindent {\sc\small
Alfr\' ed R\' enyi Institute of Mathematics, \\
Hungarian Academy of Sciences, \\
1364 Budapest, Hungary}\\
E-mail: {\tt revesz@renyi.hu}

\end{document}